\documentclass[a4paper,11pt,reqno]{amsart}
\usepackage{amsmath}
\usepackage{amssymb}
\usepackage{xcolor}

\usepackage{mathrsfs}
\renewcommand{\mathcal}{\mathscr}

\theoremstyle{plain} 
\newtheorem{theorem}{Theorem}[section]

\newtheorem{lemma}[theorem]{Lemma}

\newtheorem{remark}[theorem]{Remark}

\makeatletter
    
    \@addtoreset{equation}{section}
  \makeatother

\title[ERW in the triangular array setting]{The elephant random walk in the triangular array setting}
\thanks{R.R. thanks Keio University for its hospitality during multiple visits. M.T. and H.T. thank the Indian Statistical Institute for its hospitality. M.T. is partially supported by JSPS KAKENHI Grant Numbers JP19H01793, JP19K03514 and JP22K03333. H.T. is partially supported by JSPS KAKENHI Grant Number JP19K03514, JP21H04432 and JP23H01077.
}
\author{Rahul Roy}
\address{Indian Statistical Institute, New Delhi, India}
\email{rahul@isid.ac.in}
\author{Masato Takei}
\address{Department of Applied Mathematics, Faculty of Engineering, Yokohama National University, Yokohama, Japan}
\email{takei-masato-fx@ynu.ac.jp}
\author{Hideki Tanemura}
\address{Department of Mathematics, Keio University, Yokohama, Japan}
\email{tanemura@math.keio.ac.jp}

\begin{document}

\begin{abstract} 
Gut and Stadm\"{u}ller (2021, 2022) initiated the study of the elephant random walk with limited memory. Aguech and El Machkouri (2024) published a paper in which they discuss an extension of results by Gut and Stadtm\"{u}ller (2022) for an ``increasing memory" version of the elephant random walk without stops.
Here we present a formal definition of the process which has been hinted at Eq.~(2.2) in Gut and Stadtm\"{u}ller (2022). This definition is based on the triangular array setting. We give a positive answer to the open problem in Gut and Stadtm\"{u}ller (2022) for the elephant random walk, possibly with stops.
We also obtain 
the CLT for the supercritical case of this model.
\end{abstract}

\maketitle

\section{Introduction}
\label{sec:intro}

In recent years there has been a lot of interest in the study of the elephant random walk (ERW) since it was introduced by Sch\"{u}tz and Trimper \cite{SchutzTrimper04}. See the excellent thesis of Laulin \cite{Laulin22PhD} for a detailed bibliography. The standard ERW is described as follows. 
Let $p \in (0,1)$ and $s \in [0,1]$. We consider a sequence $X_1,X_2,\ldots$ of random variables taking values in $\{+1, -1\}$ given by
\begin{align}
X_1 = \begin{cases}
+1 & \text{with probability } s \\
-1 & \text{with probability } 1-s,
\end{cases}
\label{def:ERWfirststep}
\end{align}
$\{U_n: n \geq 1\}$ a sequence of independent random variables, independent of $X_1$, with $U_n$ having a uniform distribution over $\{1,  \ldots , n\}$ and, for $n \in \mathbb{N}:=\{ 1,2,\ldots\}$,
\begin{align}
 X_{n+1} = \begin{cases}
+X_{U_n} & \text{with probability } p\\
-X_{U_n} & \text{with probability } 1-p.
\end{cases}
\label{def:ERWNextsteps}
\end{align}
The ERW $\{W_n\}$ is defined by 
\begin{align}
W_n= \sum_{k=1}^n X_k \quad \mbox{for $n\in \mathbb{N}$.}
\label{def:ERWposition}
\end{align}

Gut and Stadm\"{u}ller \cite{GutStadtmuller21JAP,GutStadtmuller22SPL} studied a variation of this model which is described in Section 3.2 of \cite{GutStadtmuller21JAP}. Aguech and El Machkouri \cite{AguechElMachkouri24} also studied a similar variation of the model which is described in Section 2 of \cite{AguechElMachkouri24}.
We present two different formalization of the model. Gut and Stadm\"{u}ller \cite{GutStadtmuller22SPL}, Aguech and El Machkouri \cite{AguechElMachkouri24}, and our work are based on the first formalization. 

\noindent {\bf TRIANGULAR ARRAY SETTING.} 

Consider a sequence $\{m_n : n \in \mathbb{N}\}$ of positive integers satisfying 
\begin{align}
1 \leq m_n \leq n \quad \mbox{for each $n\in \mathbb{N}$.}
\label{cond:m_nGen}
\end{align}
Let $X_1,X_2,\ldots$ be the sequence defined by \eqref{def:ERWfirststep} and \eqref{def:ERWNextsteps}.
We define a triangular array of random variables $\{\{S^{(n)}_k : 1 \leq k \leq n\} : n \in \mathbb{N}\}$ as follows.
Let $\{Y^{(n)}_k : 1 \leq k \leq n\}$ be random variables with
\begin{align}
Y^{(n)}_k = \begin{cases}
X_{k} & \text{ for } 1 \leq k \leq m_n\\
X^{(n)}_k & \text{ for } m_n < k \leq n,
\end{cases}
\label{def:ERWincrsteps}
\end{align}
where, for $m_n < k \leq n$,
\begin{align}
X^{(n)}_k = \begin{cases}
+X_{U_{k, n}} & \text{with probability } p\\
-X_{U_{k, n}} & \text{with probability } 1-p.
\end{cases}
\label{def:ERWincrLatersteps}
\end{align}
Here ${\mathcal U}_n := \{U_{k, n}: m_n < k \leq n\}$ is an i.i.d. collection of uniform random variables over $\{1, \ldots , m_n\}$ and $\{{\mathcal U}_n: n \in \mathbb{N}\}$ is an independent collection.
Finally, for $1 \leq k \leq n$ let
\begin{align}
S^{(n)}_k := \sum_{i=1}^k Y^{(n)}_i.
\label{def:ERWincrAtTimek}
\end{align}
We note that for fixed $n \in \mathbb{N}$, the sequence $\{S^{(n)}_k: 1 \leq k \leq n\}$ is a random walk with increments in $\{+1, -1\}$. However the sequence $\{S^{(n)}_n: n \in \mathbb{N}\}$ does not have such a representation.
We study properties of the sequence $\{T_n : n\in \mathbb{N}\}$ given by
\begin{align} \label{eq:TnDef}
T_n := S^{(n)}_n.
\end{align}
Gut and Stadtm\"{u}ller \cite{GutStadtmuller22SPL} called the process $\{T_n : n\in \mathbb{N}\}$ the {\it ERW with gradually increasing memory}, where
\begin{align}
\lim_{n \to \infty} m_n =+\infty.
\label{cond:m_nIncreasing}
\end{align}

\noindent {\bf  LINEAR SETTING.} 

In this setting the ERW $W'_{n+1} := W'_{n} + Z_{n+1}$ is given by the 
increments 
\begin{align}
W'_1 = Z_{1} = \begin{cases}
+1 & \text{with probability } s \\
-1 & \text{with probability } 1-s,
\end{cases}
\end{align}
and 
\begin{align}
Z_{n+1} =  \begin{cases}
+Z_{V_n} & \text{with probability } p \\
-Z_{V_n} & \text{with probability } 1-p,
\end{cases}
\label{def:ERWincrLaterstepsAnother}
\end{align}
where $V_n$ is a uniform random variable over $\{1, \ldots , m_n\}$, and $\{V_n : n \in \mathbb{N}\}$ is an independent collection.

\begin{remark} \label{rem:Remark1.1}
We note here that the dependence structure in the definitions of $T_n$ and $W'_n$ are different and as such results obtained for $T_n$ need not carry to those obtained for $W'_n$. The error in Theorem 2 (3) of Aguech and El Machkouri \cite{AguechElMachkouri24} is due to the use of the linear setting for their equation (3.20), while working on the triangular array setting. In particular, there is a mistake in the expression of $\overline{M}_{\infty}$ in page 14 in \cite{AguechElMachkouri24}.
\end{remark}

In the next section we present the statements of our results, and in Sections \ref{sec:ProofTheoremCLT} and \ref{sec:ProofTheoremsLLN} we prove the results.
In Section \ref{sec:ERWStop} and thereafter,
we study similar questions about the ERW with stops and present our results. 

\section{Results for the ERW in the triangular array setting}
\label{sec:DefERWincrmemory}

Before we state our results, we give a short synopsis of the results  for the standard ERW $\{W_n\}$, see \cite{BaurBertoin16,Bercu18,Collettietal17a,Collettietal17b,GuerinLaulinRaschel23,KubotaTakei19JSP,Qin23}.
Let $\alpha:=2p-1$. 
\smallskip

\noindent
$\bullet$ For $\alpha \in (-1,1)$, 
\begin{align}
\lim_{n \to \infty} \dfrac{W_n}{n} = 0\quad\mbox{a.s. and in $L^2$.} \label{eq:ERWordinarySLLN}
\end{align}
$\bullet$ For $\alpha \in (-1,1/2)$,
\begin{align}
&\dfrac{W_n}{\sqrt{n}} \stackrel{\text{d}}{\to} N\left(0,\dfrac{1}{1-2\alpha}\right) \quad \mbox{as $n \to \infty$,}\label{eq:ERWordinarySubCLT} \\
&\limsup_{n \to \infty} \pm \dfrac{W_n}{\sqrt{2 n \log \log n}} =  \dfrac{1}{\sqrt{1-2\alpha}}\quad \mbox{a.s.} \label{eq:ERWordinarySubLIL}
\end{align}
$\bullet$ For $\alpha=1/2$,
\begin{align}
&\dfrac{W_n}{\sqrt{n \log n}} \stackrel{\text{d}}{\to} N(0,1) \quad \mbox{as $n \to \infty$,}
\label{eq:ERWordinaryCriticalCLT} \\
&\limsup_{n \to \infty} \pm \dfrac{W_n}{\sqrt{2 n\log n \log \log \log n}} = 1 \quad \mbox{a.s.} \label{eq:ERWordinaryCriticalLIL}
\end{align}
$\bullet$ For $\alpha \in (1/2,1)$, there exists a random variable $M$ s.t.
\begin{align}
\lim_{n \to \infty}\dfrac{W_n}{n^{\alpha}} = M \quad \mbox{a.s. and in $L^2$,} 
\label{eq:ERWordinarySuperLimit}
\end{align}
where
\[ E[M] = \dfrac{\beta}{\Gamma(\alpha+1)} ,\quad E[M^2]>0,\quad P(M \neq 0)=1, \]
and
\begin{align}
\dfrac{W_n-Mn^{\alpha}}{\sqrt{n}} \stackrel{\text{d}}{\to} N\left(0,\dfrac{1}{2\alpha-1}\right) \quad \mbox{as $n \to \infty$.} 
\label{eq:ERWordinarySuperCLT}
\end{align}


Our first result improves and extends Theorem 3.1 in \cite{GutStadtmuller22SPL}.

\begin{theorem} \label{thm:CLT} Let $p \in (0,1)$ and $\alpha=2p-1$. Assume that $\{m_n : n \in \mathbb{N}\}$ satisfies \eqref{cond:m_nGen}, \eqref{cond:m_nIncreasing} and 
\begin{align}
 \gamma_n := \dfrac{m_n}{n} \quad \mbox{and} \quad \lim_{n \to \infty} \gamma_n = \gamma \in [0,1].
 \label{cond:m_nStability}
\end{align}
\noindent (i) If $\alpha \in (-1,1/2)$ then
\begin{align}
\dfrac{\gamma_n T_n}{\sqrt{m_n}}
 \stackrel{\text{d}}{\to} N\left(0,\dfrac{\{\gamma+\alpha(1-\gamma)\}^2}{1-2\alpha}+\gamma(1-\gamma)\right) \quad \mbox{as $n \to \infty$.}
\label{thm:CLT(i)}
\end{align}
(ii) If $\alpha =1/2$ then 
\begin{align}
\dfrac{\gamma_n T_n}{\sqrt{m_n\log m_n}}
\stackrel{\text{d}}{\to} N\left(0,\dfrac{(1+\gamma)^2}{4}\right) \quad \mbox{as $n \to \infty$.} \label{thm:CLT(ii)}
\end{align}
(iii) If $\alpha \in (1/2,1)$ then
\begin{align}
\lim_{n \to \infty} \dfrac{\gamma_n T_n}{(m_n)^{\alpha}}
=\{\gamma+\alpha(1-\gamma)\}M
 \quad\mbox{a.s. and in $L^2$,}
 \label{thm:CLT(iii)a}
\end{align}
where $M$ is the random variable in \eqref{eq:ERWordinarySuperLimit}.
Moreover,
\begin{align}
&\dfrac{\gamma_n T_n- M \cdot \{\gamma_n + \alpha(1-\gamma_n)\} \cdot (m_n)^{\alpha}}{\sqrt{m_n}} \notag \\
&\stackrel{\text{d}}{\to} N\left(0,\dfrac{\{\gamma+\alpha(1-\gamma)\}^2}{2\alpha-1}+\gamma(1-\gamma)\right) \quad \mbox{as $n \to \infty$.} \label{thm:CLT(iii)b}
\end{align}
\end{theorem}

\begin{remark} 
If $\alpha=\gamma=0$ then the right hand side of \eqref{thm:CLT(i)} is $N(0,0)$, which we interpret as the delta measure at $0$. 

Our result \eqref{thm:CLT(iii)b} differs from the Theorem 2 (3) of Aguech and El Machkouri \cite{AguechElMachkouri24}.
The reason for this is given in Remark \ref{rem:Remark1.1}.
\end{remark}

The next theorem is concerning the strong law of large numbers and its refinement.

\begin{theorem} \label{thm:sLLN}  Let $p \in (0,1)$ and $\alpha=2p-1$. Assume that $\{m_n : n \in \mathbb{N}\}$ satisfies \eqref{cond:m_nGen}, \eqref{cond:m_nIncreasing} and \eqref{cond:m_nStability}. Then we have
\begin{align}
 \lim_{n \to \infty}  \dfrac{T_n}{n} = 0  \quad \mbox{a.s.}
\label{statement:sLLN}
\end{align}
Actually we obtain the following sharper result: If $c \in (\max\{\alpha,1/2\},1)$ then
\begin{align}
\lim_{n \to \infty} 
\dfrac{\gamma_n T_n}{(m_n)^c}
= 0 \quad \mbox{a.s.} 
\label{statement:sLLNsharpened}
\end{align}
\end{theorem} 

\section{Proof of Theorem \ref{thm:CLT}}
\label{sec:ProofTheoremCLT}

Throughout this section we assume that \eqref{cond:m_nGen}, \eqref{cond:m_nIncreasing} and \eqref{cond:m_nStability} hold.

Let $\mathcal{F}_n$ be the $\sigma$-algebra generated by $X_1,\ldots,X_n$. 
For $n \in \mathbb{N}$, the conditional distribution of $X_{n+1}$ given the history up to time $n$ is
\begin{align}
 &P(X_{n+1}= \pm 1 \mid \mathcal{F}_n) \notag \\
 &= \dfrac{\#\{k=1,\ldots,n : X_k=\pm 1\}}{n} \cdot p+  \dfrac{\#\{k=1,\ldots,n : X_k= \mp 1\}}{n} \cdot (1-p) \notag \\
&=
\dfrac{1}{2} \left( 1 \pm \alpha \cdot \dfrac{W_n}{n} \right).
\label{eq:elephantRWCondDistp}
\end{align}

For each $n \in \mathbb{N}$, let
\begin{align*}
 \mathcal{G}_{m_n}^{(n)} = \mathcal{F}_{\infty} &:= \sigma(\{X_i : i \in \mathbb{N} \}) 
=\sigma(\{X_1\} \cup \{ U_i : i \in \mathbb{N} \}) 
\end{align*}
and
\begin{align*}
\mathcal{G}_k^{(n)} &:= \sigma(\{X_i : i \in \mathbb{N} \} \cup \{ X_i^{(n)} : m_n<i \leq k\}) \\
&=\sigma(\{X_1\} \cup \{ U_i : i \in \mathbb{N} \} \cup \{ U_{i,n} : m_n<i \leq k\})
\end{align*}
for $k \in (m_n,n] \cap \mathbb{N}$.
From \eqref{def:ERWincrLatersteps}, we can see that the conditional distribution of $X_k^{(n)}$ for $k \in (m_n,n] \cap \mathbb{N}$ is given by
\begin{align}
 P (X_k^{(n)} = \pm 1 \mid \mathcal{G}_{k-1}^{(n)}) 
 &= \dfrac{1}{2} \left( 1 \pm \alpha \cdot \dfrac{W_{m_n}}{m_n} \right). \label{eq:elephantRWIncrMemoryCondDistp}
\end{align}
(This corresponds to Eq. (2.2) in \cite{GutStadtmuller22SPL}.)
From \eqref{eq:elephantRWIncrMemoryCondDistp} we have that
\begin{align} \label{eq:E[X_k^n]}
 E[ X_k^{(n)} \mid \mathcal{F}_{\infty}]  
 &= \alpha \cdot \dfrac{W_{m_n}}{m_n}
\end{align}
for each $k \in (m_n,n] \cap \mathbb{N}$, and
\begin{align} \label{eq:CondE[center]}
E[T_n-W_{m_n} \mid \mathcal{F}_{\infty}] = \sum_{k=m_n+1}^{n} E[ X_k^{(n)} \mid \mathcal{F}_{\infty}] = \alpha(n-m_n) \cdot  \dfrac{W_{m_n}}{m_n}.
\end{align}



We introduce
\begin{align}
A_n := E[T_n \mid \mathcal{F}_{\infty}] 
\quad \mbox{and} \quad
B_n := T_n-A_n. 
\label{eq:AnBnDecomp}
\end{align}
Noting that
\begin{align} 
A_n =W_{m_n} + E[T_n-W_{m_n} \mid \mathcal{F}_{\infty}] = \dfrac{W_{m_n}}{\gamma_n} \cdot \{ \gamma_n + \alpha(1-\gamma_n) \}, 
\label{eq:AnKey}
\end{align}
we have
\begin{align}
\gamma_n T_n = \gamma_n (A_n + B_n) = c_n W_{m_n}+\gamma_n B_n,
\label{eq:GammanTnDecompose2}
\end{align}
where
\begin{align}
c_n=c_n(\alpha):=\gamma_n+\alpha(1-\gamma_n).
\label{eq:cnDefinition}
\end{align}

First we prove Theorem \ref{thm:CLT} (i). Assume that $\alpha \in (-1,1/2)$. By \eqref{eq:AnKey} and \eqref{eq:ERWordinarySubCLT}, we have that
\begin{align*}
\dfrac{\gamma_n A_n}{\sqrt{m_n}} 
= \dfrac{c_n W_{m_n}}{\sqrt{m_n}}
&\stackrel{\text{d}}{\to}  \{  \gamma + \alpha (1-\gamma)\}  \cdot N\left(0,\dfrac{1}{1-2\alpha}\right)\quad \mbox{as $n \to \infty$.}
\end{align*}
In terms of characteristic functions, this is equivalent to for $\xi \in \mathbb{R}$, 
\begin{align}
E\left[ \exp\left( \dfrac{i\xi \gamma_n A_n}{\sqrt{m_n}}\right) \right] \to \exp\left( -\dfrac{\xi^2}{2} \cdot \dfrac{\{\gamma+\alpha(1-\gamma)\}^2}{1-2\alpha} \right) \quad \mbox{as $n \to \infty$.}
\label{eq:ProofTheoremCLTAn}
\end{align}
Now we turn to $\{B_n\}$. Unless specified otherwise all the results  on $\{ B_n \}$ hold for all $\alpha \in (-1,1)$. Since for each $n \in \mathbb{N}$, $X_k^{(n)}$ for $k \in (m_n,n] \cap \mathbb{N}$ are independent, identically distributed under $P(\,\cdot\,\mid \mathcal{F}_{\infty})$, so
\begin{align}
 B_n=\sum_{k=m_n+1}^n\{ X_k^{(n)}-E[X_k^{(n)}\mid \mathcal{F}_{\infty}]\}
\label{def:Bn}
\end{align}
is a sum of centered i.i.d. random variables.
The conditional variance of $X_k^{(n)}$ for
\begin{align}
V[ X_k^{(n)} \mid \mathcal{F}_{\infty}] &= E[ (X_k^{(n)})^2 \mid \mathcal{F}_{\infty}] - (E[ X_k^{(n)} \mid \mathcal{F}_{\infty}])^2 \notag \\
&=\begin{cases}
0 &\mbox{for $k \in [1,m_n] \cap \mathbb{N}$} \\
1 - \alpha^2 \cdot \left(\dfrac{W_{m_n}}{m_n}\right)^2 &\mbox{for $k \in (m_n,n] \cap \mathbb{N}$}.
\end{cases}
\label{eq:CondVarXkn} 
\end{align}
We have
\begin{align}
V\left[ \left. \dfrac{\gamma_n  B_n}{\sqrt{m_n}} \,\right| \,\mathcal{F}_{\infty} \right] 
&= \dfrac{(\gamma_n)^2}{m_n}\cdot (n-m_n)\cdot \left\{1 - \alpha^2 \cdot \left(\dfrac{W_{m_n}}{m_n}\right)^2\right\}  \notag \\
&= \gamma_n (1-\gamma_n) \cdot \left\{1 - \alpha^2 \cdot \left(\dfrac{W_{m_n}}{m_n}\right)^2\right\}, \label{eq:CondVarBnNormalized}  
\end{align}
which converges to $\gamma(1-\gamma)$  as $n \to \infty$ a.s. by \eqref{eq:ERWordinarySLLN}.
Based on this observation, we prove the following.

\begin{lemma} \label{lem:ProofTheoremCLTBn} For 
$\gamma \in [0,1]$, 
\begin{align}
E\left[ \left. \exp\left( \dfrac{i\xi \gamma_n B_n}{\sqrt{m_n}}\right) \,\right|\, \mathcal{F}_{\infty} \right] \to \exp\left(-\dfrac{\xi^2}{2} \cdot \gamma(1-\gamma) \right)\quad \mbox{as $n \to \infty$ a.s.}
\label{eq:ProofTheoremCLTBn}
\end{align}
\end{lemma}

\begin{proof} Because $B_n$ is the sum \eqref{def:Bn} of centered i.i.d. random variables under $P(\,\cdot\,\mid \mathcal{F}_{\infty})$, by \eqref{eq:CondVarXkn} we have
\begin{align*}
&E\left[ \left. \exp\left( \dfrac{i\xi \gamma_n B_n}{\sqrt{m_n}}\right) \,\right|\, \mathcal{F}_{\infty} \right] \\
&= \left[1 -  \dfrac{\xi^2 \gamma_n}{2n} \cdot \left\{1 - \alpha^2 \cdot \left(\dfrac{W_{m_n}}{m_n}\right)^2\right\} + o \left(\dfrac{\gamma_n}{n}\right)\right]^{n-m_n}\quad \mbox{as $n \to \infty$ a.s.}
\end{align*}
Note that 
\[ 
 \dfrac{\gamma_n}{n} \to 0
 \quad\mbox{and}\quad 
 \dfrac{\gamma_n}{n}\cdot (n-m_n) =\gamma_n(1-\gamma_n) \to \gamma(1-\gamma) \]
as $n \to \infty$. Now \eqref{eq:ProofTheoremCLTBn} follows from this  together with \eqref{eq:ERWordinarySLLN}.
\end{proof}
From \eqref{eq:GammanTnDecompose2}, \eqref{eq:ProofTheoremCLTAn} and \eqref{eq:ProofTheoremCLTBn} together with the bounded convergence theorem, we can see that
\begin{align*}
E\left[ \exp\left( \dfrac{i\xi \gamma_n T_n}{\sqrt{m_n}}\right)  \right] 
= E\left[ \exp\left( \dfrac{i\xi \gamma_n A_n}{\sqrt{m_n}}\right) \cdot E\left[ \left. \exp\left( \dfrac{i\xi \gamma_n B_n}{\sqrt{m_n}}\right) \,\right|\, \mathcal{F}_{\infty} \right] \right] 
\end{align*}
converges to
\begin{align*}
\exp\left( -\dfrac{\xi^2}{2} \cdot \dfrac{\{\gamma+\alpha(1-\gamma)\}^2}{1-2\alpha} \right)\cdot \exp\left(-\dfrac{\xi^2}{2} \cdot \gamma(1-\gamma) \right) 
\end{align*}
as $n \to \infty$. This gives \eqref{thm:CLT(i)}.

The proof of Theorem \ref{thm:CLT} (ii) is along the same lines as that of (i), and is actually simpler. Assume that $\alpha =1/2$. As $c_n(1/2)=(1+\gamma_n)/2$, from \eqref{eq:AnKey} and \eqref{eq:ERWordinaryCriticalCLT}, we have that
\begin{align*}
\dfrac{\gamma_n A_n}{\sqrt{m_n\log m_n}} &= \dfrac{c_n W_{m_n}}{\sqrt{m_n \log m_n}}  
\stackrel{\text{d}}{\to}  \dfrac{1+\gamma}{2} \cdot N\left(0,1\right)\quad \mbox{as $n \to \infty$.}
\end{align*}
Also, from \eqref{eq:CondVarBnNormalized} and \eqref{eq:ERWordinarySLLN}, we have that
\begin{align}
E\left[ \left(\dfrac{\gamma_n B_n}{\sqrt{m_n}}\right)^2 \right] &= \gamma_n (1-\gamma_n) \left\{1 - \alpha^2 \cdot E\left[\left(\dfrac{W_{m_n}}{m_n}\right)^2\right] \right\} \notag \\ 
&\to \gamma(1-\gamma) \quad  \mbox{as $n \to \infty$.} \label{asymp:E[Bn]Normalized}
\end{align}
This implies that $\gamma_n B_n/\sqrt{m_n \log m_n} \to 0$ as $n \to \infty$ in $L^2$. 
By Slutsky's lemma,  we obtain \eqref{thm:CLT(ii)}.

Finally we prove Theorem \ref{thm:CLT} (iii). Assume that $\alpha \in (1/2,1)$. By \eqref{eq:AnKey} and \eqref{eq:ERWordinarySuperLimit},
\begin{align}
\dfrac{\gamma_n A_n}{(m_n)^{\alpha}} &=  \dfrac{c_n W_{m_n}}{(m_n)^{\alpha}} 
\to  \{\gamma+\alpha(1-\gamma)\}M
\label{eq:AnSuperdiffusiveASL2}
\end{align}
as $n \to \infty$ a.s. and in $L^2$. Noting that 
$\gamma_n B_n/(m_n)^{\alpha} \to 0$ as $n \to \infty$ in $L^2$, 
by \eqref{asymp:E[Bn]Normalized}, we obtain the $L^2$-convergence in
\eqref{thm:CLT(iii)a}.
From \eqref{eq:ProofTheoremSLLNBnCruder} which will be proved in the next section, the almost-sure convergence in \eqref{thm:CLT(iii)a} follows. 
To prove \eqref{thm:CLT(iii)b}, by \eqref{eq:GammanTnDecompose2}, we have
\begin{align}
\gamma_n T_n-c_n \cdot M \cdot (m_n)^{\alpha} = c_n \{W_{m_n}-M \cdot (m_n)^{\alpha}\} + \gamma_n B_n.
\label{eq:GammanTnDecompose}
\end{align}
Note that $M$ is $\mathcal{F}_{\infty}$-measurable. Using \eqref{eq:ERWordinarySuperCLT}, \eqref{eq:ProofTheoremCLTBn} and \eqref{eq:GammanTnDecompose}, we obtain \eqref{thm:CLT(iii)b} similarly as in the proof of \eqref{thm:CLT(i)}.
\qed

\section{Proof of Theorem \ref{thm:sLLN}}
\label{sec:ProofTheoremsLLN}

In this section we assume that \eqref{cond:m_nGen}, \eqref{cond:m_nIncreasing} and \eqref{cond:m_nStability} hold.

First we give almost-sure bounds for $\{B_n\}$.

\begin{lemma} \label{lem:ProofTheoremSLLNBn} For any $\alpha \in (-1,1)$ and $\gamma \in [0,1]$,
\begin{align}
\limsup_{n \to \infty} \dfrac{\gamma_n B_n}{\sqrt{2\gamma_n(1-\gamma_n)m_n \log n}} \leq 1 \quad \mbox{a.s.}
\label{eq:ProofTheoremSLLNBn}
\end{align}
In particular, for any $c \in (1/2,1)$,
\begin{align}
\lim_{n \to \infty} \dfrac{\gamma_n B_n}{(m_n)^c} = 0 \quad \mbox{a.s.}
\label{eq:ProofTheoremSLLNBnCruder}
\end{align}
\end{lemma}

\begin{proof} Note that $|X_k^{(n)}-E[X_k^{(n)}\mid \mathcal{F}_{\infty}]| \leq 1$ for each $1 \leq k \leq n$. For $\lambda \in \mathbb{R}$, it follows from Azuma's inequality \cite{Azuma67} that
\begin{align*}
&E[ \exp(\lambda \gamma_n B_n) \mid \mathcal{F}_{\infty} ] \notag \\&=E\left[\left. \exp\left(\lambda \gamma_n \sum_{k=m_n+1}^n\{ X_k^{(n)}-E[X_k^{(n)}\mid \mathcal{F}_{\infty}]\}\right)\,\right|\, \mathcal{F}_{\infty} \right] \notag \\
&\leq \exp((\lambda \gamma_n)^2 (n-m_n)/2),
\end{align*}
and
\begin{align*}
P(|\gamma_n B_n| \geq x ) \leq 2\exp\left(-\dfrac{x^2}{2\gamma_n (1-\gamma_n)m_n}\right) \quad \mbox{for $x > 0$}.
\end{align*}
Taking $x=\sqrt{2(1+\varepsilon) \gamma_n (1-\gamma_n)m_n \log n}$ for some $\varepsilon>0$, we have
\begin{align*}
\sum_{n=1}^{\infty} P(|\gamma_n B_n| \geq \sqrt{2(1+\varepsilon) \gamma_n (1-\gamma_n)m_n \log n}) \leq \sum_{n=1}^{\infty} \dfrac{2}{n^{1+\varepsilon}} .
\end{align*}
This together with the Borel--Cantelli lemma implies \eqref{eq:ProofTheoremSLLNBn}. To obtain \eqref{eq:ProofTheoremSLLNBnCruder}, note that for $c \in (1/2,1)$,
\begin{align*}
\dfrac{2\gamma_n(1-\gamma_n)m_n \log n}{(m_n)^{2c}} &= \dfrac{2(1-\gamma_n) \log n}{n(m_n)^{2c-2}} 
\leq \dfrac{2(1-\gamma_n) \log n}{n^{2c-1}} \to 0 \quad \mbox{as $n \to \infty$,}
\end{align*}
where we used $2c-2<0<2c-1$ and $m_n \leq n$. 
\end{proof}


We prove \eqref{statement:sLLNsharpened} in Theorem \ref{thm:sLLN}. 
Eq. \eqref{statement:sLLN} is readily derived from \eqref{statement:sLLNsharpened}.
For the case $\alpha \in (-1,1/2)$, 
from \eqref{eq:ERWordinarySubLIL} and \eqref{eq:AnKey} we have that
\begin{align*}
\limsup_{n \to \infty} \dfrac{\gamma_n A_n}{\sqrt{2m_n \log \log m_n}} \leq \dfrac{\gamma + \alpha(1-\gamma)}{\sqrt{1-2\alpha}}\quad \mbox{a.s.} %
\end{align*}
For the case $\alpha=1/2$,
from \eqref{eq:ERWordinaryCriticalLIL} and \eqref{eq:AnKey} we have that
\begin{align*}
\limsup_{n \to \infty} \dfrac{\gamma_n A_n}{\sqrt{2m_n\log m_n \log \log \log m_n}} \leq \dfrac{1+\gamma}{2}
 \quad \mbox{a.s.} 
\end{align*}
By \eqref{eq:ProofTheoremSLLNBnCruder}, if $\alpha \in (-1,1/2]$ then \eqref{statement:sLLNsharpened} holds for any $c \in (1/2,1)$.
As for the case $\alpha \in (1/2,1)$, almost sure convergence in \eqref{thm:CLT(iii)a} follows from \eqref{eq:AnSuperdiffusiveASL2} and \eqref{eq:ProofTheoremSLLNBnCruder}. Thus \eqref{statement:sLLNsharpened} holds for any $c \in (\alpha,1)$.
\qed
%

\section{The ERW with stops in the triangular array setting}
\label{sec:ERWStop}

Let $s \in [0,1]$, and assume that $p,q,r \in [0,1)$ satisfies $p+q+r=1$. In this section we consider the {\it ERW with stops} $\{W_n\}$
whose increments are given by
the sequence $X_1,X_2,\ldots$ of random variables taking values in $\{+1, -1\}$ defined by \eqref{def:ERWfirststep} and 
\begin{align}
X_{n+1} &=  \begin{cases}
X_{U_n} &\mbox{with probability $p$} \\
-X_{U_n} &\mbox{with probability $q$} \\
0 &\mbox{with probability $r$.} \\
\end{cases}
\label{def:ERWwithStopsNextsteps}
\end{align}
Note that if $r=0$ then it is the standard ERW defined in Section \ref{sec:intro}. Hereafter we assume that $r \in (0,1)$.

The ERW with stops was introduced by Kumar {\it et al.} \cite{KumarHarbolaLindenberg10PRE}. To describe the limit theorems 
obtained by Bercu \cite{Bercu22}, it is convenient to introduce the following new parameters:
\begin{align*}
\alpha:=p-q\quad \mbox{and} \quad \beta:=1-r,
\end{align*}
where $\beta \in (0,1)$ and $\alpha \in [-\beta,\beta]$.
Let $\Sigma_n$ be the number of moves up to time $n$, that is
\begin{align}
\Sigma_n := \sum_{k=1}^n 1_{\{X_k \neq 0\}} = \sum_{k=1}^n X_k^2 \quad \mbox{for $n\in \mathbb{N}$}.
\end{align}
It is shown in \cite{Bercu22} that
\begin{align}
\lim_{n \to \infty} \dfrac{\Sigma_n}{n^{\beta}} = \Sigma>0 \quad \mbox{a.s. and in $L^2$,} \label{eq:ERWSSigmanLimit}
\end{align}
where $\Sigma$ has a Mittag--Leffler distribution with parameter $\beta$. We turn to the CLT for $\{W_n \}$ in \cite{Bercu22}
\smallskip

\noindent
$\bullet$ For $\alpha \in [-\beta,\beta/2)$,
\begin{align}
\frac{W_n}{\sqrt{\Sigma_n}}\stackrel{\text{d}}{\to}N\left(0,\dfrac{\beta}{\beta-2\alpha}\right)\quad \mbox{as $n \to \infty$.}
\label{eq:Bercu22JSP(3.8)}
\end{align}
$\bullet$ For $\alpha=\beta/2$,
\begin{align}
\dfrac{W_n}{\sqrt{\Sigma_n \log \Sigma_n}}\stackrel{\text{d}}{\to}N(0,1)\quad \mbox{as $n \to \infty$.} 
\label{eq:Bercu22JSP(3.16)}
\end{align}
$\bullet$ For $\alpha \in (\beta/2,\beta]$, there exists a random variable $M$ s.t.
\begin{align}
\lim_{n \to \infty} \dfrac{W_n}{n^{\alpha}} = M \quad \mbox{a.s. and in $L^2$,} 
\label{eq:Bercu22JSP(3.18)}
\end{align}
and
\begin{align}
\dfrac{W_n-Mn^{\alpha}}{\sqrt{\Sigma_n}} \stackrel{\text{d}}{\to} N\left(0,\dfrac{\beta}{2\alpha-\beta}\right) \quad \mbox{as $n \to \infty$,} 
\label{eq:Bercu22JSP(3.26)}
\end{align}
where $P(M>0)>0$.

Next we define the sequence $\{T_n\}$ as in \eqref{eq:TnDef} however $Y_k^{(n)}$ and $X_k^{(n)}$ of \eqref{def:ERWincrsteps} and \eqref{def:ERWincrLatersteps} are defined with $\{X_i\}$ as in \eqref{def:ERWwithStopsNextsteps} instead of \eqref{def:ERWNextsteps}.
We call it the {\it ERW with stops in the triangular array setting}. 

Our first result of this section is an extension of Theorem 4.1 in \cite{GutStadtmuller22SPL}. We note here that \cite{GutStadtmuller22SPL} allows $X_1$ to take value $0$ with probability $r$ unlike in this paper. As such they have an extra $\delta_0$ in their results for the case $\gamma=0$.

\begin{theorem} \label{thm:StopCLT} Let $\beta \in (0,1)$ and $\alpha \in [-\beta,\beta]$. Assume that $\{m_n : n \in \mathbb{N}\}$ satisfies \eqref{cond:m_nGen}, \eqref{cond:m_nIncreasing} and \eqref{cond:m_nStability}. 

\noindent (i) If $\alpha \in [-\beta,\beta/2)$ then
\begin{align}
\dfrac{\gamma_n T_n}{\sqrt{\Sigma_{m_n}}} 
 \stackrel{\text{d}}{\to} N\left(0,\dfrac{\beta \{\gamma+\alpha(1-\gamma)\}^2}{\beta-2\alpha}+\beta \gamma(1-\gamma)\right) \quad \mbox{as $n \to \infty$.}
\label{thm:StopCLT(i)}
\end{align}
(ii) If $\alpha = \beta/2$ then 
\begin{align}
\dfrac{\gamma_n T_n}{\sqrt{\Sigma_{m_n} \log \Sigma_{m_n}}}
\stackrel{\text{d}}{\to}  N\left(0,\{\gamma+\beta(1-\gamma)/2\}^2 \right) \quad \mbox{as $n \to \infty$.} 
\label{thm:StopCLT(ii)}
\end{align}
(iii) If $\alpha \in (\beta/2,\beta]$ then
\begin{align}
\lim_{n \to \infty} \dfrac{\gamma_n T_n}{(m_n)^{\alpha}}
=\{\gamma+\alpha(1-\gamma)\}M
 \quad \mbox{in $L^2$,}\label{thm:StopCLT(iii)a} 
\end{align}
where $M$ is the random variable in \eqref{eq:Bercu22JSP(3.18)}.
Moreover,
\begin{align}
&\dfrac{\gamma_n T_n- M \cdot \{\gamma_n + \alpha(1-\gamma_n)\} \cdot (m_n)^{\alpha}}{\sqrt{\Sigma_{m_n}}} \notag \\
&\stackrel{\text{d}}{\to} N\left(0,\dfrac{\beta\{\gamma+\alpha(1-\gamma)\}^2}{2\alpha-\beta}+\beta\gamma(1-\gamma)\right) \quad \mbox{as $n \to \infty$.} \label{thm:StopCLT(iii)b}
\end{align}
\end{theorem}

\begin{remark} Unlike the results in \cite{GutStadtmuller22SPL}, we have a random normalization in the results above. This is because we consider the general case $\gamma \in [0,1]$.
We can obtain the $L^4$-convergence in \eqref{thm:StopCLT(iii)a}, using Burkholder's inequality as in Eq. (3.15) of \cite{AguechElMachkouri24}.
\end{remark}


We also consider the process $\{\Xi_n : n \in \mathbb{N} \}$ defined by
\begin{align}
\Xi_n:= \sum_{k=1}^n \{X_k^{(n)}\}^2 \quad \mbox{for $n \in \mathbb{N}$.}
\end{align}
The next theorem is an improvement of Theorem 4.2 in \cite{GutStadtmuller22SPL}.

\begin{theorem} \label{thm:LLNCLTERWSincr} Under the same condition as in Theorem \ref{thm:StopCLT}, we have that
\begin{align}
\lim_{n \to \infty} \dfrac{\gamma_n \Xi_n}{(m_n)^{\beta}} = \{\gamma + \beta(1-\gamma) \}\Sigma \quad \mbox{in $L^2$,}
\label{eq:LLNCLTERWSincr}
\end{align}
where $\Sigma$ is defined in \eqref{eq:ERWSSigmanLimit}. 
\end{theorem}

The strong law of large numbers and its refinement can be obtained also for the ERW with stops.

\begin{theorem} \label{thm:sLLN_ERWS} Under the same condition as in Theorem \ref{thm:StopCLT}, we have \eqref{statement:sLLN}. In addition, \eqref{statement:sLLNsharpened} holds for $c \in (\max\{\alpha,1/2\},1)$.
\end{theorem} 

\begin{remark} Assume that $\beta \in (1/2,1)$. As a byproduct of the proof of Theorem \ref{thm:sLLN_ERWS}, we can prove the a.s. convergence in \eqref{eq:LLNCLTERWSincr}. The a.s. convergence in  \eqref{thm:StopCLT(iii)a} is valid for $\alpha \in (1/2,\beta]$.
\end{remark}

\section{Proof of Theorem \ref{thm:StopCLT}}
\label{sec:ProofStop1}


Noting that $p=(\beta+\alpha)/2$ and $q=(\beta-\alpha)/2$, for $n \in \mathbb{N}$, we have
\begin{align}
 &P(X_{n+1}= \pm 1 \mid \mathcal{F}_n) \notag \\
 &= \dfrac{\#\{k=1,\ldots,n : X_k=\pm 1\}}{n} \cdot p+  \dfrac{\#\{k=1,\ldots,n : X_k= \mp 1\}}{n} \cdot q \notag \\
&= \dfrac{1}{2} \left( \beta \cdot \dfrac{\Sigma_n}{n} \pm \alpha \cdot \dfrac{W_n}{n} \right).
\label{eq:elephantRWCondDistpStops}
\end{align}
For $k \in (m_n,n] \cap \mathbb{N}$, we have that 
\begin{align}
P(X_k^{(n)}=\pm 1 \mid \mathcal{G}_{k-1}^{(n)}) = \dfrac{1}{2} \left(\beta \cdot \dfrac{\Sigma_{m_n}}{m_n} \pm \alpha \cdot \dfrac{W_{m_n}}{m_n}\right), \label{eq:elephantRWSTOPIncrMemoryCondDistp}
\end{align}
and
\begin{align}
P(\{X_k^{(n)}\}^2= 1 \mid \mathcal{G}_{k-1}^{(n)}) 
= \beta \cdot \dfrac{\Sigma_{m_n}}{m_n}.
\label{eq:elephantRWSTOPIncrMemorySQCondDist}
\end{align}
From \eqref{eq:elephantRWSTOPIncrMemoryCondDistp}, we see that Eqs. \eqref{eq:E[X_k^n]} and \eqref{eq:CondE[center]} continue to hold in this setting. Defining $\{A_n\}$ and $\{B_n\}$ by \eqref{eq:AnBnDecomp}, we note that they satisfy Eqs. \eqref{eq:AnKey} and \eqref{eq:GammanTnDecompose2}.

We prepare a lemma about $\{B_n\}$.

\begin{lemma} \label{lem:BnStop} Under the assumption of Theorem \ref{thm:StopCLT}, we have the following: \\
(i) For $\alpha \in [-\beta,\beta]$ and $\xi \in \mathbb{R}$,
\begin{align}
E\left[ \left. \exp\left( \dfrac{i\xi \gamma_n B_n}{\sqrt{\Sigma_{m_n}}}\right) \,\right|\, \mathcal{F}_{\infty} \right] \to \exp\left(-\dfrac{\xi^2}{2} \cdot \beta \gamma(1-\gamma) \right)\quad \mbox{as $n \to \infty$ a.s.,}
\label{eq:ProofTheoremCLTStopSubCritBn}
\end{align}
and
\begin{align}
E\left[ \left. \exp\left( \dfrac{i\xi \gamma_n B_n}{\sqrt{\Sigma_{m_n}\log \Sigma_{m_n}}}\right) \,\right|\, \mathcal{F}_{\infty} \right] \to 1\quad \mbox{as $n \to \infty$ a.s.}
\label{eq:ProofTheoremCLTStopCritBn}
\end{align}
(ii) If $\alpha \in (\beta/2,\beta]$ then $\gamma_n B_n/(m_n)^{\alpha}  \to 0$ as $n \to \infty$ in $L^2$.
\end{lemma}

\begin{proof}
Note that $E[ B_n \mid \mathcal{F}_{\infty}] = 0$. By \eqref{eq:elephantRWSTOPIncrMemoryCondDistp} and \eqref{eq:elephantRWSTOPIncrMemorySQCondDist}, \begin{align}
V[ X_k^{(n)} \mid \mathcal{F}_{\infty}] 
&= \beta \cdot \dfrac{\Sigma_{m_n}}{m_n}- \alpha^2 \cdot \left(\dfrac{W_{m_n}}{m_n}\right)^2
\label{eq:CondVarXknStop} 
\end{align}
for $k \in (m_n,n] \cap \mathbb{N}$. As in \eqref{eq:CondVarBnNormalized}, we have
\begin{align}
V[  B_n \mid \mathcal{F}_{\infty} ] 
&= (n-m_n) \cdot  \left\{\beta \cdot \dfrac{\Sigma_{m_n}}{m_n} - \alpha^2 \cdot \left(\dfrac{W_{m_n}}{m_n}\right)^2\right\}  \notag \\
&=  \dfrac{1-\gamma_n}{\gamma_n} \cdot \left\{\beta \Sigma_{m_n}  - \alpha^2 \cdot \dfrac{(W_{m_n})^2}{m_n}\right\}. \label{eq:CondVarB_nStop}  
\end{align}
From this, 
\begin{align}
V\left[ \left. \dfrac{\gamma_n B_n}{\sqrt{\Sigma_{m_n}}} \,\right| \,\mathcal{F}_{\infty} \right] 
&=\gamma_n(1-\gamma_n) \cdot  \left\{\beta   - \alpha^2 \cdot \dfrac{(m_n)^{\beta}}{\Sigma_{m_n}} \cdot \left( \dfrac{W_{m_n}}{(m_n)^{(1+\beta)/2}} \right)^2 \right\}.
\label{eq:CondVarB_nSigmaNormalizedStop} 
\end{align}
For any $\beta \in (0,1)$ and $\alpha \in [-\beta,\beta]$, we show that
\begin{align}
\lim_{n \to \infty} \dfrac{W_{m_n}}{(m_n)^{(1+\beta)/2}} = 0 \quad \mbox{a.s.}
\label{eq:CondVarB_nLemma}
\end{align}
Indeed, if $\alpha \in [-\beta,\beta/2)$ then
\begin{align}
\limsup_{n \to \infty} \dfrac{W_n}{\sqrt{2 n^{\beta} \log \log n}} = \sqrt{\dfrac{\beta \Sigma}{\beta-2\alpha}}\quad \mbox{a.s.}
\label{eq:Bercu22JSP(3.5)}
\end{align}
by Eq. (3.5) in \cite{Bercu22}. If $\alpha=\beta/2$ then 
\begin{align}
\limsup_{n \to \infty} \dfrac{W_n}{\sqrt{2 n^{\beta} \log n \log \log \log n}} = \sqrt{\beta \Sigma}\quad \mbox{a.s.}
\label{eq:Bercu22JSP(3.13)}
\end{align}
by Eq. (3.13) in \cite{Bercu22}. If $\alpha \in (\beta/2,\beta]$ then $W_{m_n}/(m_n)^{\alpha} \to M$ as $n \to \infty$ a.s. by \eqref{eq:Bercu22JSP(3.18)}, and $(1+\beta)/2 > \alpha$ since $2\alpha - \beta\leq \beta <1$.
In any case  we have \eqref{eq:CondVarB_nLemma}.
Since $(m_n)^{\beta}/\Sigma_{m_n}\to 1/\Sigma$ as $n \to \infty$ a.s. by \eqref{eq:ERWSSigmanLimit}, we see that \eqref{eq:CondVarB_nSigmaNormalizedStop} converges to $\beta \gamma(1-\gamma)$ as $n \to \infty$ a.s., and
\begin{align*}
V\left[ \left. \dfrac{\gamma_n B_n}{\sqrt{\Sigma_{m_n} \log \Sigma_{m_n}}} \,\right| \,\mathcal{F}_{\infty} \right] \to 0 \quad \mbox{as $n \to \infty$ a.s.}
\end{align*}
By a similar computation as in Lemma \ref{lem:ProofTheoremSLLNBn}, we obtain \eqref{eq:ProofTheoremCLTStopSubCritBn} and \eqref{eq:ProofTheoremCLTStopCritBn} in (i).

Next we consider (ii). 
By \eqref{eq:CondVarB_nStop}, 
\begin{align*}
E\left[ \left(\dfrac{\gamma_n B_n}{(m_n)^{\alpha}}\right)^2 \right] 
&= \gamma_n (1-\gamma_n) \cdot \left\{\beta \cdot \dfrac{E[\Sigma_{m_n}]}{(m_n)^{2\alpha}}  - \alpha^2 \cdot \dfrac{E[(W_{m_n})^2]}{(m_n)^{1+2\alpha}}\right\}. 
\end{align*}
From Eq. (A.6) in \cite{Bercu22}, $E[(W_n)^2] \sim n^{2\alpha}/\{(2\alpha-\beta)\Gamma(2\alpha)\}$ as $n \to \infty$.
On the other hand, from Eq. (4.4) in \cite{Bercu22} we can see that $E[\Sigma_n] \sim n^{\beta}/\Gamma(1+\beta)$ as $n \to \infty$.
Noting that $\beta<2\alpha$, we have (ii). 
\end{proof}

Assume that $\alpha \in [-\beta,\beta/2)$. By \eqref{eq:AnKey} and \eqref{eq:Bercu22JSP(3.8)}, we have that
\begin{align*}
\dfrac{\gamma_n A_n}{\sqrt{\Sigma_{m_n}}} = \dfrac{c_nW_{m_n}}{\sqrt{\Sigma_{m_n}}}  
\stackrel{\text{d}}{\to}  \{  \gamma + \alpha (1-\gamma)\}  \cdot N\left(0,\dfrac{\beta}{\beta-2\alpha}\right)\quad \mbox{as $n \to \infty$.}
\end{align*}
Combining this and \eqref{eq:ProofTheoremCLTStopSubCritBn},
we can prove \eqref{thm:StopCLT(i)} by the same method as for \eqref{thm:CLT(i)}. Next we consider the case $\alpha=\beta/2$. By \eqref{eq:AnKey} and \eqref{eq:Bercu22JSP(3.16)}, we have that
\begin{align*}
\dfrac{\gamma_n A_n}{\sqrt{\Sigma_{m_n}\log \Sigma_{m_n}}} &= \dfrac{c_nW_{m_n}}{\sqrt{\Sigma_{m_n}\log \Sigma_{m_n}}}  
\stackrel{\text{d}}{\to} \{  \gamma + \alpha (1-\gamma)\}  \cdot   N\left(0,1\right)\quad \mbox{as $n \to \infty$.}
\end{align*}
This together with \eqref{eq:ProofTheoremCLTStopCritBn}
gives \eqref{thm:StopCLT(ii)}.
As for the case $\alpha \in (\beta/2,\beta]$, by \eqref{eq:AnKey} and \eqref{eq:Bercu22JSP(3.18)},
\begin{align*}
\dfrac{\gamma_n A_n}{(m_n)^{\alpha}} &=  \dfrac{c_nW_{m_n}}{(m_n)^{\alpha}}
\to  \{\gamma+\alpha(1-\gamma)\}M \quad \mbox{as $n \to \infty$ a.s. and in $L^2$.}
\label{eq:AnSuperdiffusiveASL2Stop}
\end{align*}
Now \eqref{thm:StopCLT(iii)a} follows from Lemma \ref{lem:BnStop} (ii). The proof of \eqref{thm:StopCLT(iii)b} is almost identical to that of \eqref{thm:CLT(iii)b}: Use \eqref{eq:GammanTnDecompose2}, \eqref{eq:Bercu22JSP(3.26)}, and 
\eqref{eq:ProofTheoremCLTStopSubCritBn}.


\section{Proof of Theorem \ref{thm:LLNCLTERWSincr}}
\label{sec:ProofStop2}

Put $A'_n:=E[\Xi_n \mid \mathcal{F}_{\infty}]$ and $B'_n:= \Xi_n - A'_n$. 
Using \eqref{eq:elephantRWSTOPIncrMemorySQCondDist}, we can see that $\gamma_n A'_n = c_n(\beta) \Sigma_{m_n}$, which together with \eqref{eq:ERWSSigmanLimit} imply
\begin{align*}
\dfrac{\gamma_n A'_n}{(m_n)^{\beta}} &= 
\dfrac{c_n(\beta) \Sigma_{m_n}}{(m_n)^{\beta}}\to \{ \gamma+\beta(1-\gamma)\} \cdot \Sigma \quad \mbox{as $n \to \infty$ a.s. and in $L^2$.}
\end{align*}
As for $B'_n$, again by \eqref{eq:elephantRWSTOPIncrMemorySQCondDist} we can see that
\begin{align*}
V\left[ \left. \left(\dfrac{\gamma_n B'_n}{(m_n)^{\beta}}\right)^2 \,\right| \, \mathcal{F}_{\infty}\right] &= \dfrac{(\gamma_n)^2}{(m_n)^{2\beta}} \cdot \sum_{k=m_n+1}^n V[\{ X_k^{(n)}\}^2 \mid \mathcal{F}_{\infty}] \\
&= \dfrac{(\gamma_n)^2}{(m_n)^{2\beta}} \cdot (n-m_n) \cdot \beta \cdot \dfrac{\Sigma_{m_n}}{m_n} \cdot \left( 1- \beta \cdot \dfrac{\Sigma_{m_n}}{m_n} \right),
\end{align*}
and
\begin{align*}
E\left[ \left(\dfrac{\gamma_n B'_n}{(m_n)^{\beta}}\right)^2\right] &= \dfrac{\beta \gamma_n(1-\gamma_n)}{(m_n)^{\beta}} \cdot E\left[ \dfrac{\Sigma_{m_n}}{(m_n)^{\beta}}\cdot \left( 1- \beta \cdot \dfrac{\Sigma_{m_n}}{m_n} \right)\right].
\end{align*}
Since $\beta<1$ and $\Sigma_{m_n}/(m_n)^{\beta}$ converges to $\Sigma$ in $L^2$ by \eqref{eq:ERWSSigmanLimit}, we have 
\begin{align*}
E\left[ \dfrac{\Sigma_{m_n}}{(m_n)^{\beta}}\cdot \left( 1- \beta \cdot \dfrac{\Sigma_{m_n}}{m_n} \right)\right] 
&= E\left[ \dfrac{\Sigma_{m_n}}{(m_n)^{\beta}} \right] - \dfrac{\beta}{(m_n)^{1-\beta}} \cdot E\left[ \left(\dfrac{\Sigma_{m_n}}{(m_n)^{\beta}} \right)^2\right]  \\
&\to E[\Sigma] \quad \mbox{as $n \to \infty$.}
\end{align*}
Noting that $\beta>0$, this shows that $\gamma_n B'_n/(m_n)^{\beta} \to 0$ as $n \to \infty$ in $L^2$,
which completes the proof. 

\section{Proof of Theorem \ref{thm:sLLN_ERWS}}

The proof of Lemma \ref{lem:ProofTheoremSLLNBn} is based on the fact $|X_k^{(n)}-E[X_k^{(n)}\mid \mathcal{F}_{\infty}]| \leq 1$. Thus $\{B_n\}$ for the ERW with stops in the triangular array setting also satisfies \eqref{eq:ProofTheoremSLLNBnCruder} for any $c \in (1/2,1)$. If $\alpha \in [-\beta,\beta/2]$ then from \eqref{eq:AnKey}, \eqref{eq:Bercu22JSP(3.5)}, and \eqref{eq:Bercu22JSP(3.13)}, we can see that $\gamma_n A_n=o(n^c)$ for any $c \in (\beta/2,1)$. If $\alpha \in (\beta/2,\beta]$ then \eqref{eq:AnKey} and \eqref{eq:Bercu22JSP(3.18)} imply that $\gamma_n A_n=o(n^c)$ for any $c \in (\alpha,1)$. In any case \eqref{statement:sLLNsharpened} holds for $c \in (\max\{\alpha,1/2\},1)$.





\end{document}